\documentclass[a4paper,11pt]{amsart}
\usepackage[utf8]{inputenc}
\usepackage{amsmath,amssymb,amsfonts}
\usepackage{epsfig}
\usepackage{mathrsfs}
\usepackage{graphicx,color}
\usepackage{eucal}
\usepackage{bbm}

\newtheorem{thm}{Theorem}[section]

\newtheorem{lma}{Lemma}[section]

\newcommand{\norm}[1]{\left| \! \left| #1\right| \!\right|}
\newcommand{\bnorm}[1]{\big| \! \big| #1\big| \!\big|}

\newcommand{\bm}[1]{ \begin{bmatrix} #1 \end{bmatrix}}
\newcommand{\Xs}{\mathcal{X}}
\newcommand{\Ys}{\mathcal{Y}}

\newcommand{\ip}[1]{\left< #1 \right>}
\newcommand{\dom}{\mathscr{D}}

\newcommand{\cu}{\textup{cos}}
\newcommand{\su}{\textup{sin}}

\numberwithin{equation}{section}
\font\eka=cmex10
\usepackage{ae}
\def\ind{\mathrel{\hbox{\rlap{%
\hbox to 7.5pt{\hrulefill}}\raise6.6pt\hbox{\eka\char'167}}}}
\begin{document}
\title[Output error minimizing BFN method]{Output error minimizing back and forth nudging method for initial state recovery*\footnote{*T\lowercase{his is the preprint version of the article published in} S\lowercase{ystems} \& C\lowercase{ontrol} L\lowercase{etters, 94, p. 111--117 (2016).} T\lowercase{he article doi is 10.1016/j.sysconle.2016.06.002}}}


\author[Atte Aalto]{Atte Aalto \\ \\ I\lowercase{nria}, U\lowercase{niversit\'e} P\lowercase{aris}--S\lowercase{aclay}, P\lowercase{alaiseau}, F\lowercase{rance}; M$\Xi$DISIM \lowercase{team}}




\thanks{Email: atte.ej.aalto@gmail.com} 

\begin{abstract}
We show that for linear dynamical systems with skew-adjoint generators, the initial state estimate given by the back and forth nudging method with colocated feedback, converges to the minimizer of the discrepancy between the measured and simulated outputs --- given that the observer gains are chosen suitably and the system is exactly observable. If the system's generator $A$ is essentially skew-adjoint and dissipative (with not too much dissipation), the colocated feedback has to be corrected by the operator $e^{At}e^{A^*t}$ in order to obtain such convergence. In some special cases, a feasible approximation for this operator can be found analytically. The case with wave equation with constant dissipation will be demonstrated.

\medskip


\end{abstract}

\maketitle

\section{Introduction}

This paper deals with the problem of retrieving the initial state of a --- possibly infinite-dimensional --- linear dynamical system from the noisy output measurements of the system over a given, finite time interval $[0,T]$. A more or less classical approach is to minimize the quadratic discrepancy between the measured and modeled outputs over all possible initial states. This approach is often called \emph{variational data assimilation} --- for details and references, see \cite{LeDimet} by Le Dimet \emph{et al.} and \cite{Teng07} by Teng \emph{et al}. In the case of a linear system, this approach leads to a linear-quadratic optimization problem, whose solution amounts to computing and inverting the observability Gramian. 
This approach is seemingly simple, but when the system's dimension is high, the optimization task may be numerically challenging, so alternative methods are called for.

One alternative is  the \emph{back and forth nudging} (BFN) method, introduced by Auroux and Blum in \cite{AB05} and \cite{AB08}. The method is based on using a Luenberger observer alternately forward and backward in time over and over again. In these papers the theory is developed for finite-dimensional systems and it is assumed that the full state is observed. The generalization to infinite-dimensional systems and more general observation operators is presented by Ramdani \emph{et al} in \cite{Ramdani10}. There it is shown that in the absence of any noise terms, the BFN method converges exponentially to the true initial state. They assume that the system is exponentially stabilizable both to forward and backward directions. 
The BFN method is presented and reviewed in Section~\ref{sec:BFN}.

Whereas the variational method gives equal weight to all measurements on the time interval $[0,T]$, the BFN method emphasizes the measurements, and hence also measurement noise, closer to the initial time, in particular if the observer gain is high. The sensitivity to noise is expected to reduce when the gain is reduced. In Section~\ref{sec:skew}, we show that for systems with skew-adjoint generators, the initial state estimate given by the BFN method  with colocated feedback, converges to the minimizer of the discrepancy between the measured and simulated outputs --- given that the observer gains are taken to zero with a suitable rate. Systems with essentially skew-adjoint and dissipative (ESAD) generators, that is, $\dom(A^*)=\dom(A)$ and $A+A^*=-Q$ for some bounded and small enough $Q \ge 0$, are treated in Section~\ref{sec:ESAD}. Then
  the colocated feedback has to be corrected by the operator $e^{At}e^{A^*t}$ in order to obtain such convergence (without this the BFN method converges to a biased estimate). In some special cases, this operator, or a feasible approximation for it, can be found analytically. In section~\ref{sec:wave}, we demonstrate that for the wave equation with constant dissipation, $u_{tt}=\Delta u -\epsilon u_t$ with Dirichlet boundary conditions, it holds that $e^{At}e^{A^*t} \approx e^{-\epsilon t}I$ resulting in a simple discounting factor for the observer gain.  We shall also give upper bounds for the error due to the approximation $k(t)I \approx e^{At}e^{A^*t}$ in the observer gain. These bounds are given in the presented wave equation context, but the results hold more generally.

In the paper we use notation $\mathcal{L}(\mathcal{H}_1,\mathcal{H}_2)$ for the space of bounded linear operators from a Hilbert space $\mathcal{H}_1$ to another Hilbert space $\mathcal{H}_2$. We also denote $\mathcal{L}(\mathcal{H})=\mathcal{L}(\mathcal{H},\mathcal{H})$. When there is no possibility of confusion, the notation $\norm{\cdot}$ is used without indication in which space the norm is computed.

\section{Problem setup and the back and forth nudging method} \label{sec:BFN}

Consider the problem of retrieving the initial state of the system
\[
\begin{cases}
\dot z=Az+f+\eta, \\
z(0)=z_0, \\
y=Cz+\nu
\end{cases}
\]
from measurements $y(t)$ for $t \in [0,T]$. Here $A:\Xs \to \Xs$ is the generator of a dissipative semigroup $e^{At}$ on the state space $\Xs$ with domain $\dom(A)$. The output operator $C:\Xs \to \Ys$ is assumed to be bounded and both $\Xs$ and $\Ys$ are assumed to be separable Hilbert spaces. The load term $f$ is assumed to be known and $\eta$ and $\nu$ are unknown input and output noise terms, respectively. Of the load and noise terms we only assume that they are smooth enough so that $y \in L^2(0,T;\Ys)$. 

The back and forth nudging method is defined as follows. The dynamics of the \emph{forward observer} for $j = 1,2,...$ are governed by
\begin{equation} \label{eq:fwd}
\begin{cases}
\dot z_j^+(t)=Az_j^+(t)+f(t)+\kappa_jC^*(y(t)-Cz_j^+(t)), \\
z_j^+(0)=z_{j-1}^-(T), \quad \textrm{for } j \ge 2.
\end{cases}
\end{equation}
For $j=1$, the initial state can be any vector in $\Xs$, since its contribution will vanish. The \emph{backward observer} is also defined "forward in time"
\begin{equation} \label{eq:bwd}
\begin{cases}
\dot z_j^-(t)=-Az_j^-(t)-f(T-t)+\kappa_jC^*(y(T-t)-Cz_j^-(t)), \\
z_j^-(0)=z_j^+(T),
\end{cases}
\end{equation}
that is, $z_j^-(t)$ is an estimate of $z(T-t)$ and the initial state estimate that we are interested in is given by $z_j^-(T)$.
The feedback of the form $C^*(y-Cz)$ in the observers is called \emph{colocated feedback}, roughly meaning that the measurement through $C$ and the control action through $C^*$ take place in the same physical location in the computational domain. Classical references on the colocated feedback are \cite{Liu} by Liu for skew-adjoint operators and \cite{Curtain_Weiss} by Curtain and Weiss studying also ESAD operators. For a study on the colocated feedback for the wave equation, see \cite{Chapelle_wave} by Chapelle \emph{et al}.

We show for systems with skew-adjoint generators, that if the observer gains $\kappa_j$ in the back and forth nudging iterations \eqref{eq:fwd} and \eqref{eq:bwd} are selected in a certain way, then the initial state estimate will converge to the minimizer of the cost function
\begin{equation} \label{eq:cost_init}
J(x):=\frac12 \int_0^T \norm{y(s)-Cz[x](s)}^2ds
\end{equation}
where $z[x]$ is the solution of
\begin{equation} \label{eq:zinit}
\begin{cases}
\dot z[x]=Az[x]+f, \\
z[x](0)=x.
\end{cases}
\end{equation}
Complementary results are obtained for systems with ESAD generators and for the classical BFN method with constant feedback $\kappa_j=\kappa$. 

In the first results on the BFN method, \cite{AB05} and \cite{AB08}, the feedback term is simply a matrix $K$ that can be chosen freely. Obviously it can be chosen so that both $A-K$ and $-A-K$ have strictly negative eigenvalues. Then if there are no noises, the BFN algorithm converges exponentially to the true initial state. The article \cite{Ramdani10} lays the foundation for the algorithm for infinite-dimensional systems. There the feedback in the observers is of the form $\pm A - K_{\pm}C$ where $C$ is a given (possibly unbounded) observation operator and the feedback operator $K_{\pm}$ can be chosen freely. The main result itself is similar as that of  \cite{AB05} and \cite{AB08}, namely exponential convergence to the true initial state if $K_{\pm}$ can be chosen so that $\pm A - K_{\pm}C$ generate exponentially stable semigroups, and if the output is not corrupted by noise. Numerical aspects of the method are considered by Haine and Ramdani in \cite{HR12}. The BFN method for systems with skew-adjoint operators with colocated feedback is studied by Ito \emph{et al.} in \cite{Ito11}  and by Phung and Zhang in \cite{Focus}. In the latter article the method is called \emph{time reversal focusing} and they treat the concrete problem of retrieving the initial state of the Kirchhoff plate equation from partial field measurements. Further development of the BFN method includes \cite{Haine} by Haine showing a partial convergence result when the exact observability assumption is not satisfied, and \cite{Fridman} by Fridman extending the result to a class of semilinear systems. Application to unbounded computational domain is considered by Fliss \emph{et al.} in \cite{Seb}, and  a variant for systems containing a diffusive term is suggested by Auroux \emph{et al.} in \cite{ABN11} where the idea is to change the sign of the diffusive term in the backward phase. The effect of input and output noise on the method has been briefly discussed by Shim \emph{et al.} in \cite{Shim} and by Donovan \emph{et al.} in \cite{Donovan}. The BFN method or a related time-reversal approach can also be used for source identification problems, as in \cite{Ammari13} by Ammari \emph{et al.}

\section{Results}

We shall start by showing an important lemma. In the most general cases treated in this paper, we have feedbacks of the form $A-\kappa K(r)C^*C$ for the forward observer and $-A-\kappa K(T-r)C^*C$ for the backward observer where $K(\cdot) \in C(0,T;\mathcal{L}(\Xs))$. We remark that when $A$ is ESAD, then also $-A$ generates a strongly continuous semigroup since it can be viewed as a bounded perturbation of a skew-adjoint operator $A_0^*=-A-Q/2$ (see \cite[Sections~1.10 and~3.1]{Pazy}). For any $x \in \Xs$, it holds that
\[
\frac{d}{dt}\norm{e^{-At}x}^2=\ip{Qe^{-At}x,e^{-At}x}\le\norm{Q}\norm{e^{-At}x}^2
\]
and so by Gr\"onwall's inequality, $\norm{e^{-At}} \le e^{\norm{Q/2}t}$.

Since also $K(r)C^*C$ is bounded, the operators $A-\kappa K(r)C^*C$  and $-A-\kappa K(T-r)C^*C$ generate strongly continuous time evolution operators $U^+(t,s)$ and $U^-(t,s)$, respectively (see \cite[Section~5.2]{Pazy}). Define also $U^{\pm}(t)=U^{\pm}(t,0)$.
As will be seen later in the proofs of our main results, after every forward and backward iteration, the old error term is multiplied by $U^-(T)U^+(T)$. We now show that if the dissipative term $Q$ is small enough, and if $K(t) \approx k(t)I$ for some strictly positive function $k(\cdot)$, then this operator is strictly contractive.
\begin{lma} \label{lma:stable}
Assume that the system is exactly observable at time $T$, that is, $\int_0^T \norm{Ce^{At}x}^2dt \ge \delta \norm{x}^2$ for all $x \in \Xs$ and some $\delta >0$.
Assume also $\dom(A^*)=\dom(A)$ and $A+A^*=-Q$ with $Q \ge 0$,  and that there exists a function $k \in C(0,T)$ with $k_1 \ge k(t) \ge k_0>0$, so that $Q$, $K(t)$, and $k(t)$ satisfy
\begin{align*}
\alpha:=&2k_0\delta -2\norm{C}^2 \! \Bigg( \! 2k_1 \!\! \left(\frac{e^{\norm{Q/2}T} \! -1}{\norm{Q/2}}-T \!\right) \! + \! \int_0^T \! e^{\norm{Q/2}s}\norm{K(s)-k(s)I}ds\Bigg) \! >0.
\end{align*}
Then 
$ \displaystyle
\ \norm{U^-(T)U^+(T)}_{\mathcal{L}(\Xs)} \le 1-\alpha \kappa +\mathcal{O}(\kappa^2).
$
\end{lma}
\noindent In the special case $Q=0$, the lemma with $K(t)=I$ suffices. Notice that in this case the result holds with $\alpha=2\delta$. In addition, by similar techniques, it can be shown separately for the forward and backward operators that $\norm{e^{(\pm A-\kappa C^*C)T}}\le 1-\delta\kappa+\mathcal{O}(\kappa^2)$.
\begin{proof}
The semigroup perturbation formula (see \cite[Section~3.1]{Pazy}) is easily checked also for time-dependent perturbations, and it gives
\begin{align*}
U^-(T)U^+(T) =&\left(e^{-AT}-\kappa \int_0^T e^{-A(T-s)}K(T-s)C^*CU^-(s)ds \right) \times \\ & \times
\left(e^{AT}-\kappa \int_0^Te^{A(T-s)}K(s)C^*CU^+(s)ds \right).
\end{align*}
From this equation it is possible to collect the zeroth and first order terms at $\kappa=0$ to get
\begin{align} \nonumber
& U^-(T)U^+(T) \\ \nonumber &=I-2\kappa \int_0^T e^{-As}K(s)C^*Ce^{As}ds+\mathcal{O}(\kappa^2) \\ \label{eq:approx} &= I-2\kappa \int_0^T k(s) e^{-As}C^*Ce^{As}ds \\ \nonumber & \ \ \ -2\kappa  \int_0^T  e^{-As}\big(K(s)-k(s)I\big)C^*Ce^{As}ds+\mathcal{O}(\kappa^2). \hspace{-5mm}
\end{align}
Applying the perturbation formula again for $e^{-As}=e^{(A^*+Q)s}$ gives
\begin{equation} \label{eq:pert}
e^{(A^*+Q)s}=e^{A^*s}+\int_0^s e^{A^*(s-r)}Qe^{(A^*+Q)r}dr.
\end{equation}
Recalling $\norm{e^{(A^*+Q)r}} \le e^{\norm{Q/2}r}$, we get a bound for the second term in~\eqref{eq:pert}:
\begin{equation} \label{eq:gronwall}
\norm{\int_0^s e^{A^*(s-r)}Qe^{(A^*+Q)r}dr} \le \norm{Q} \int_0^s \bnorm{e^{(A^*+Q)r}}dr \le 2(e^{\norm{Q/2}s}-1).
\end{equation}
The third term in \eqref{eq:approx} can be bounded by $2\kappa \norm{C}^2 \int_0^T e^{\norm{Q/2}s}\norm{K(s)-k(s)I}ds$.
Using this bound and equations \eqref{eq:approx}--\eqref{eq:gronwall}, we have
\begin{align*}
&\norm{U^-(T)U^+(T)}_{\mathcal{L}(\Xs)} \\ & \le \norm{ I-2k_0\kappa \int_0^T e^{A^*s}C^*Ce^{As}ds}+4\kappa k_1\norm{C}^2\left( \frac{e^{\norm{Q/2}T}-1}{\norm{Q/2}}-T\right) \\ & \ \ \ +2\kappa \norm{C}^2 \int_0^T e^{\norm{Q/2}s}\norm{K(s)-k(s)I}ds +\mathcal{O}(\kappa^2)
\end{align*}
where the replacement of $k(s)$ by its lower bound $k_0$ is justified by positivity of the term $e^{A^*s}C^*Ce^{As}$. The operator in the first term on the right hand side is self-adjoint and positive-definite (for $\kappa$ small enough), so its norm can be bounded using the observability assumption by
\[
\norm{ I-2k_0\kappa \int_0^T e^{A^*s}C^*Ce^{As}ds} \le 1-2k_0\delta\kappa
\]
completing the proof.
\end{proof}

\subsection{Systems with skew-adjoint generator} \label{sec:skew}

We now move on to prove the first main result of the paper, namely the convergence result in the case of a system with skew-adjoint generator.
\begin{thm} \label{thm:skew}
Assume $\dom(A^*)=\dom(A)$ and $A+A^*=0$. Assume also $C \in \mathcal{L}(\Xs,\Ys)$ and $\int_0^T \norm{Ce^{At}x}^2 dt \ge \delta \norm{x}^2$ for all $x \in \Xs$ and some $\delta>0$. Choose the observer gains $\kappa_j >0$ so that
$
\sum_{j=1}^{\infty}\kappa_j = \infty$
and
$\sum_{j=1}^{\infty}\kappa_j^2 < \infty$.
Then as $j \to \infty$, the initial state estimate $z_j^+(0)$ converges strongly to the minimizer of the cost function $J$ defined in \eqref{eq:cost_init}.
\end{thm}
\begin{proof}
Due to the assumed exact observability, the cost function $J$ is strictly convex and thus a unique minimizer $x^o$ exists. The minimizer is characterized by $\nabla J(x)|_{x=x^o}=0$ (the Fr\'echet derivative of $J$ with respect to $x$), which is equivalent to
\begin{equation} \label{eq:nablaJ}
\int_0^T \ip{y(s)-Cz[x^o](s),Ce^{As}h}ds=0, \qquad \forall \, h \in \Xs
\end{equation}
where $z[x^o]$ is defined in \eqref{eq:zinit}.
Denote $y-Cz[x^o]=:\chi$ and notice that by~\eqref{eq:nablaJ}, 
\begin{equation} \label{eq:perpnoise}
\int_0^T e^{A^*s}C^*\chi(s) ds=0. 
\end{equation}
Now we can summarize
\[
\begin{cases}
\dot z[x^o](t)=Az[x^o](t)+f(t),  
\\ z[x^o](0)=x^o, \\
y(t)=Cz[x^o](t)+\chi(t).
\end{cases}
\]
Denote then $\varepsilon_j^+(t):=z[x^o](t)-z_j^+(t)$ and $\varepsilon_j^-(t):=z[x^o](T-t)-z_j^-(t)$ .
By \eqref{eq:fwd} and \eqref{eq:bwd}, they satisfy
\[
\begin{cases}
\dot\varepsilon_j^+(t)=(A-\kappa_jC^*C)\varepsilon_j^+(t)-\kappa_jC^* \chi(t), \qquad
\varepsilon_j^+(0)=\varepsilon_{j-1}^-(T), \\
\dot\varepsilon_j^-(t)=(-A-\kappa_jC^*C)\varepsilon_j^-(t)-\kappa_jC^* \chi(T-t), \qquad
\varepsilon_j^-(0)=\varepsilon_j^+(T).
\end{cases}
\]
The solution for the first equation is given by
\begin{equation} \label{eq:vareps}
\varepsilon_j^+(t)=e^{(A-\kappa_jC^*C)t}\varepsilon_{j-1}^-(T)-\kappa_j \int_0^t e^{(A-\kappa_j C^*C)(t-s)}C^*\chi(s)ds.
\end{equation}
The second term itself is a solution to
\[
\dot\varepsilon(t)=(A-\kappa_jC^*C)\varepsilon(t)-\kappa_jC^* \chi(t), \qquad
\varepsilon(0)=0,
\]
and it can be decomposed into
\begin{equation} \label{eq:decomposition}
\varepsilon(t)=-\kappa_j\int_0^t e^{A(t-s)}C^* \chi(s)ds-\kappa_j\int_0^te^{A(t-s)}C^*C\varepsilon(s)ds.
\end{equation}
The first term is zero at $t=T$ by \eqref{eq:perpnoise} and $-A=A^*$. For the second term, it can be seen directly from \eqref{eq:vareps}, that 
$\norm{\varepsilon(t)} \le \kappa_j\norm{C}\sqrt{t} \norm{\chi}_{L^2(0,T)}$.
Then from \eqref{eq:decomposition}, we get
\begin{equation} \label{eq:quad}
\norm{\varepsilon(T)} \le \frac23 \kappa_j^2 \norm{C}^3 T^{3/2} \norm{\chi}_{L^2(0,T)}.
\end{equation}
The exactly same steps can be taken with $\varepsilon_j^-$. Then, by Lemma~\ref{lma:stable} with $K(s)=I$,
\[
\bnorm{\varepsilon_j^-(T)} \le \big( 1-2\delta\kappa_j \big) \bnorm{ \varepsilon_{j-1}^-(T)}+\mathcal{O}(1) \kappa_j^2 
\]
where the $\mathcal{O}(1)$-term refers to the asymptotic behavior as $\kappa_j \to 0$ and it contains $\frac43\norm{C}^3 T^{3/2} \norm{\chi}_{L^2(0,T)}$ from \eqref{eq:quad} and the contribution of the  $\mathcal{O}(\kappa_j^2)$-term from Lemma~\ref{lma:stable}. Finally,
\begin{align*}
\bnorm{\varepsilon_j^-(T)} \le & \prod_{i=1}^j \big( 1- 2\delta\kappa_i \big) \norm{\varepsilon_1^+(0)}  +\mathcal{O}(1)\sum_{i=1}^j \kappa_i^2 \prod_{k=i+1}^j \big( 1- 2\delta\kappa_k \big)
\end{align*}
from which the convergence can be deduced using the assumptions on $\kappa_j$'s and
\[
\prod_{k=i+1}^j \big( 1- 2\delta\kappa_k \big)=\exp \left( \sum_{k=i+1}^j \ln \big( 1- 2\delta\kappa_k \big)   \right) \le \exp \left( -2\delta \! \sum_{k=i+1}^j \kappa_k \right)
\]
which converges to zero for any $i$ as $j \to \infty$.
\end{proof}

\subsection{Systems with ESAD generator} \label{sec:ESAD}

In the case the generator satisfies $A+A^*=-Q$ for $Q \ge 0$ and $Q \ne 0$ is small enough, we get the following result.
\begin{thm} \label{thm:ESAD}
Assume $\dom(A)=\dom(A^*)$ and $A+A^*=-Q$ for $Q \ge 0$ where $Q$ is a bounded operator small enough to satisfy
\[
e^{-\norm{Q/2}T}\delta-3\norm{C}^2 \left(\frac{e^{\norm{Q/2}T}-1}{\norm{Q/2}}-T \right) > 0,
\]
 and $\int_0^T \norm{Ce^{At}x}^2dt \ge \delta \norm{x}^2$. Replace the feedback operator  $\kappa_j C^*$ in the forward observer \eqref{eq:fwd} by $\kappa_j P(t)C^*$  where $P(t)=e^{At}e^{A^*t}$, and by $\kappa_j P(T-t)C^*$ in the backward observer \eqref{eq:bwd}. Assume again $\sum_{j=1}^{\infty}\kappa_j = \infty$ and $\sum_{j=1}^{\infty}\kappa_j^2 < \infty$.  Then the initial state estimate given by the back and forth nudging method  converges strongly to the minimizer of the cost function $J$.
\end{thm}
\begin{proof}
Let us first show that the assumption on $Q$ justifies the application of Lemma~\ref{lma:stable} with $k(t)=e^{-\norm{Q/2}t}$ so that $k_0=e^{-\norm{Q/2}T}$ and $k_1=1$. Application of the semigroup perturbation formula to $A^*=-A-Q$ gives
\[
e^{At}e^{A^*t}=I-\int_0^t e^{As}Qe^{A^*s}ds
\]
from which it is possible to deduce  $e^{-\norm{Q}t}I \le e^{At}e^{A^*t} \le I$. Thus \linebreak $\norm{e^{At}e^{A^*t} - e^{-\norm{Q/2}t}I} \le 1-e^{-\norm{Q/2}t}$.  Using this, we get
\[
\int_0^T e^{\norm{Q/2}s}\bnorm{e^{As}e^{A^*s} - e^{-\norm{Q/2}t}I}ds  \le \frac{e^{\norm{Q/2}T}-1}{\norm{Q/2}}-T
\]
assuring that $\alpha$ in Lemma~\ref{lma:stable} is strictly positive.

The key steps in the proof are exactly the same as in the proof of Theorem~\ref{thm:skew}, but \eqref{eq:decomposition} is modified a little to
\begin{align*}  
\varepsilon(t)=&-\kappa_j\int_0^t e^{A(t-s)}e^{As}e^{A^*s}C^* \chi(s)ds -\kappa_j\int_0^te^{A(t-s)}e^{As}e^{A^*s}C^*C\varepsilon(s)ds \\ =& -\kappa_j e^{At} \int_0^t e^{A^*s}C^* \chi(s)ds-\kappa_j e^{At} \int_0^te^{A^*s}C^*C\varepsilon(s)ds
\end{align*}
from which we proceed as before.
\end{proof}
\noindent Notice that without the correction $P(t)$ in the feedback term, the initial state estimate from the BFN method converges to $x^{\textup{bias}}$ satisfying 
\[
\int_0^T e^{-As}C^* \big( y(s)-Cz[x^{\textup{bias}}](s)\big) ds =0
\]
instead of \eqref{eq:perpnoise} which characterizes the optimum $x^o$. Here $z[x^{\textup{bias}}]$ is defined in \eqref{eq:zinit}.

 The benefit of using the BFN method lies in the computational lightness of the utilized Luenberger-type observer. Therefore it is usually not desirable to numerically compute the full operator $P(t)=e^{At}e^{A^*t}$ required in the previous theorem. Luckily, in some special cases this operator can be at least approximated analytically. The case with wave equation with constant dissipation term will be demonstrated in Section~\ref{sec:wave}. 

We remark that the assumed bound on $\norm{Q}$ in Theorem~\ref{thm:ESAD} can be quite restrictive. However, it should be viewed as a sufficient condition for the theorem, but the algorithm may convergence even if this condition is not satisfied.

\subsection{The classical BFN approach}

In the classical back and forth nudging method with colocated feedback the gain is kept constant, that is, $\kappa_j=\kappa$. We  show that in such case the BFN estimate converges to the minimizer of the cost function $J$  but with $z[x]$ defined by
\begin{equation} \label{eq:zinit_feedback}
\begin{cases}
\dot z[x]=Az[x]+f+\kappa C^*(y-Cz[x]), \\
z[x](0)=x.
\end{cases}
\end{equation}
We remark that for example in the presence of modeling errors, it may happen that the measurement $y$ cannot be even closely reproduced by the open loop system \eqref{eq:zinit} with any initial state $x$. In such case the minimizer of $J$ given by \eqref{eq:cost_init} with \eqref{eq:zinit} cannot be expected to be very good. Also, if $Q$ does not satisfy the assumptions of Theorem~\ref{thm:ESAD}, then it may not be possible to take $\kappa \to 0$.

We shall show this only in the skew-adjoint case, but a similar variant for ESAD systems is possible. In addition, we again make the exact observability assumption, but this theorem can be straightforwardly generalized to the non-observable case as is done in \cite{Haine}. In that case the convergence is not exponential and of course the minimizer is not necessarily unique.
\begin{thm}
Assume $\dom(A^*)=\dom(A)$ and $A+A^*=0$. Assume also $\int_0^T\norm{Ce^{At}x}^2dt \ge \delta \norm{x}^2$ for all $x \in \Xs$ and some $\delta >0$.
Then the initial state estimate from the BFN method with constant observer gain $\kappa$ converges exponentially to the minimizer of $J$ given in \eqref{eq:cost_init} with $z[x]$ defined in \eqref{eq:zinit_feedback}.
\end{thm}
\begin{proof}
Denote again the minimizer by $x^o$ and the corresponding solution of \eqref{eq:zinit_feedback} by $z[x^o]$. Denote again $y-Cz[x^o]=:\chi$. The minimizer is characterized by
\begin{equation} \label{eq:minchar}
\int_0^T e^{(A^*-\kappa C^*C)s}C^*\chi(s)ds=0.
\end{equation}
Denote $\varepsilon_j^+(t)=z[x^o](t)-z_j^+(t)$ and $\varepsilon_j^-(t)=z[x^o](T-t)-z_j^-(t)$. They satisfy
\[ \arraycolsep=1.4pt\def\arraystretch{1.15}
\left\{ \begin{array}{lll}
\dot \varepsilon_j^+(t) &=&(A-\kappa C^*C)\varepsilon_j^+(t), \\
\dot \varepsilon_j^-(t) &=&(-A+\kappa C^*C)z[x^o](T-t)  -(-A-\kappa C^*C)z_j^-(t)-2\kappa C^*y(T-t) \\
&=&(-A-\kappa C^*C)\varepsilon_j^-(t)-2\kappa C^* \chi(T-t).
\end{array} \right.
\]
Now $\varepsilon_j^+(T)=e^{(A-\kappa C^*C)T}\varepsilon_j^+(0)$ and
\[
\varepsilon_j^-(T)=e^{(-A-\kappa C^*C)T}\varepsilon_j^-(0)+\int_0^T e^{(-A-\kappa C^*C)(T-s)}C^*\chi(T-s)ds
\]
where the second term is zero by \eqref{eq:minchar} since $-A=A^*$. By \cite[Theorem~2.3 (c)]{Liu} or \cite[Theorem~1.1]{Curtain_Weiss}, it holds that $\norm{e^{(\pm A-\kappa C^*C)T}} \le \gamma$ with some $\gamma < 1$ and hence $\bnorm{\varepsilon_j^-(T)} \le \gamma^{2j} \bnorm{\varepsilon_1^+(0)}$.
\end{proof}

\section{Wave equation with dissipation} \label{sec:wave}

Consider the wave equation with constant dissipation
\begin{equation} \label{eq:wave}
\left\{ \!\!
\begin{array}{ll}
u_{tt}(x,t)=\Delta u(x,t) -\epsilon u_t(x,t), &  x \in \Omega, \ t \in \mathbb{R}^+, \\
u(x,t)=0,  & x \in \partial\Omega, \\
u(x,0)=u_0(x), \ u_t(x,0)=v_0(x) &
\end{array}
\right.
\end{equation}
where $\Omega \subset \mathbb{R}^n$ is a sufficiently smooth domain and $\epsilon \ge 0$. As usual, \eqref{eq:wave} is written as a first order system using $v=u_t$,
\begin{equation} \label{eq:wavesys}
\frac{d}{dt}\begin{bmatrix} u \\ v \end{bmatrix}=\begin{bmatrix} 0 & I \\ \Delta & -\epsilon I \end{bmatrix} \begin{bmatrix} u \\ v \end{bmatrix},
\end{equation} 
which is denoted $\dot z = Az$. The state space is $\Xs=H_0^1(\Omega) \times L^2(\Omega)$ where the first component is equipped with the norm $\norm{u}_{H_0^1(\Omega)}^2:=\int_{\Omega} \norm{\nabla u}^2 dx$. In this space, it holds that $A+A^*= -\left[ \begin{smallmatrix} 0&0 \\ 0& 2\epsilon I \end{smallmatrix} \right]$.

Now let $\{-\lambda_j\}_{j=1}^{\infty} \subset \mathbb{R}^-$ be the sequence of eigenvalues of the Laplacian in $\Omega$ with Dirichlet boundary conditions in ascending order (by their absolute values) and denote by $\{e_j(x)\}_{j=1}^{\infty}$ the corresponding $L^2$-normalized eigenfunctions. Assume that the dissipation satisfies $\epsilon^2 < 4\lambda_1$. By the separation of variables principle, the system's eigenfrequencies are given by  $\omega_j:=\sqrt{\lambda_j-\frac{\epsilon^2}4}$, and the solution to the initial value problem \eqref{eq:wave} is given by \[
u(x,t)=e^{-\frac{\epsilon}2t}\sum_{j=1}^{\infty} \left[ \alpha_j \left( \cu(\omega_j t) +\frac{\epsilon}{2\omega_j} \su(\omega_j t) \right) + \frac{\beta_j}{\omega_j}\su(\omega_j t) \right]e_j(x)
\]
where the coefficients $\alpha_j$ and $\beta_j$ are the Fourier coefficients of $u_0$ and $v_0$, respectively.
From this solution we can construct the semigroup $e^{At}$ in the basis $\left\{ \left[ \begin{smallmatrix} e_j(x) \\ 0 \end{smallmatrix} \right] \right\}_{j=1}^{\infty} \bigcup \left\{ \left[ \begin{smallmatrix} 0 \\ e_j(x) \end{smallmatrix} \right] \right\}_{j=1}^{\infty}$  as
\[
e^{At}=e^{-\frac{\epsilon}2t}  \left[ \begin{array}{lr} \cu(\omega_j t)+\frac{\epsilon}{2\omega_j} \su(\omega_j t) & \frac1{\omega_j} \su(\omega_jt) \vspace{2mm} \\  -\frac{\lambda_j}{\omega_j} \su(\omega_jt) & \hspace{-10mm}\cu(\omega_jt)-\frac{\epsilon}{2\omega_j}\su(\omega_jt) \end{array} \right]
\]
where the elements are interpreted as infinite diagonal matrices, $j=1,2,...$.
Note that the inner product in $\Xs$ in this basis is given by 
\begin{equation} \label{eq:ip}
\ip{ \bm{ \alpha \\ \beta} , \bm{a \\ b}}_{\!\! \Xs}=\sum_{j=1}^{\infty} \big( \lambda_j \alpha_j a_j + \beta_j b_j \big),
\end{equation}
and so we have 
\begin{align*}
&e^{At}e^{A^*t} \\ &=e^{-\epsilon t} \bm{\cu^2(\omega_jt)+\frac{\epsilon}{\omega_j}\su(\omega_jt)\cu(\omega_jt)+\frac{\lambda_j+\epsilon^2/4}{\lambda_j-\epsilon^2/4}\su^2(\omega_jt) & -\frac{\epsilon}{\omega_j^2}\su^2(\omega_jt) \hspace{-4mm} \vspace{3mm} \\ \hspace{-59mm} -\frac{\lambda_j\epsilon}{\omega_j^2}\su^2(\omega_jt) & \hspace{-55mm} \cu^2(\omega_jt)-\frac{\epsilon}{\omega_j}\su(\omega_jt)\cu(\omega_jt)+\frac{\lambda_j+\epsilon^2/4}{\lambda_j-\epsilon^2/4}\su^2(\omega_jt)}  \\
&=e^{-\epsilon t} I + \epsilon e^{-\epsilon t} \bm{\frac{1}{\omega_j}\su(\omega_jt)\cu(\omega_jt)+\frac{\epsilon/2}{\lambda_j-\epsilon^2/4}\su^2(\omega_jt)& -\frac{1}{\omega_j^2}\su^2(\omega_jt) \hspace{-4mm} \vspace{3mm} \\ \hspace{-40mm} -\frac{\lambda_j}{\omega_j^2}\su^2(\omega_jt) & \hspace{-37mm} -\frac{1}{\omega_j}\su(\omega_jt)\cu(\omega_jt)+\frac{\epsilon/2}{\lambda_j-\epsilon^2/4}\su^2(\omega_jt)}.
\end{align*}
In the following error estimates, we need a bound for the $\mathcal{L}(\Xs)$-norm of $e^{At}e^{A^*t}-e^{-\epsilon t}I$. Because of the blockwise structure of the matrix operator above, and taking into account \eqref{eq:ip}, an estimate is obtained by finding a uniform (that is, holding for all $j=1,2,...$) bound for the $\mathbb{R}^{2\times2}$ matrix norms of the blocks
\[
\bm{\frac{1}{\omega_j}\su(\omega_jt)\cu(\omega_jt)+\frac{2\epsilon}{\omega_j^2}\su^2(\omega_jt)& -\frac{\sqrt{\lambda_j}}{\omega_j^2}\su^2(\omega_jt) \hspace{-4.5mm} \vspace{2mm} \\ \hspace{-29mm} -\frac{\sqrt{\lambda_j}}{\omega_j^2}\su^2(\omega_jt) & \hspace{-27mm} -\frac{1}{\omega_j}\su(\omega_jt)\cu(\omega_jt)+\frac{2\epsilon}{\omega_j^2}\su^2(\omega_jt)}
\] 
which, in turn, can be bounded from above by the Frobenius norm, yielding
\begin{equation} \label{eq:frob}
\norm{e^{At}e^{A^*t}-e^{-\epsilon t}I} \le \epsilon e^{-\epsilon t}\frac{2\sqrt{\lambda_1}}{\lambda_1-\epsilon^2/4}.
\end{equation}

The error stemming from using $K(t)=e^{-\epsilon t}$ instead of $K(t)=e^{A^*t}e^{At}$ in the observer gain can be bounded from above:
\begin{thm} \label{thm:wave1}
Assume $C \in \mathcal{L}(\Xs,\Ys)$ is such that the system is exactly observable at time $T$. Assume also that $\epsilon$ is small enough so that the assumption of Lemma~\ref{lma:stable} is satisfied, and that $\sum_{j=1}^{\infty}\kappa_j = \infty$ and $\sum_{j=1}^{\infty}\kappa_j^2 < \infty$.

Then as $j \to \infty$, the back and forth observer with feedback $\kappa_je^{-\epsilon t} C^*$ converges to an estimate $\tilde x$, for which it holds that
\[
\norm{x^o-\tilde x} \le \frac{\epsilon\norm{C}\sqrt{T}}{\delta}\frac{2\sqrt{\lambda_1}}{\lambda_1-\epsilon^2/4}\norm{\tilde\chi}_{L^2(0,T)}
\]
where $\tilde\chi=y-Cz[\tilde x]$ and $z[\tilde x]$ is defined in \eqref{eq:zinit}.
\end{thm}
\begin{proof}
By repeating the proof of Theorem~\ref{thm:skew} with the feedback term multiplied by $k(t)=e^{-\epsilon t}$, it can be seen that the BFN method converges to $\tilde x$, which is characterized by
\[
\int_0^T e^{-As}k(s)C^*\tilde\chi(s)ds=0.
\]
Inserting here $k(s)I=P(s)-\big(P(s)-k(s)I \big)$ and recalling $P(s)=e^{As}e^{A^*s}$, we get
\[
\int_0^Te^{A^*s}C^*\tilde\chi(s)ds-\int_0^Te^{-As}\big(P(s)-k(s)I\big)C^*\tilde\chi(s)ds=0.
\]
Now combining this, equation \eqref{eq:perpnoise} characterizing the optimum $x^o$, and  $\tilde\chi(t)-\chi(t)=Ce^{At}(x^o-\tilde x)$, yields
\begin{equation} \label{eq:estimate}
\int_0^T e^{A^*s}C^*Ce^{As}(x^o-\tilde x)ds=\int_0^Te^{-As}\big(P(s)-k(s)I\big)C^*\tilde\chi(s)ds.
\end{equation}

Finally, using the observability assumption, Cauchy--Schwartz inequality, the bound \eqref{eq:frob}, and the bound $\norm{e^{-As}}\le e^{\norm{Q/2}s}=e^{\epsilon s}$, we have the result.
\end{proof}
The bound \eqref{eq:frob} for $\norm{P(s)-k(s)I}$ is based on the operator's biggest component, corresponding to the system's lowest eigenmode. However, the lowest modes are typically better observable, and hence the inverse of the observability Gramian $\int_0^T e^{A^*s}C^*Ce^{As}ds$ in \eqref{eq:estimate} is likely to suppress these modes more efficiently than with coefficient $1/\delta$ which is based on the poorly identifiable modes. Therefore the error is likely to be considerably smaller than what is obtained in the previous theorem.

The error estimate of Theorem~\ref{thm:wave1} depends on $\norm{\tilde\chi}_{L^2(0,T)}$ which makes it effectively an \emph{a posteriori} estimate. We present another error estimate, which is based on a direct computation utilizing equation \eqref{eq:decomposition}. For this we need to set $\kappa_j=\kappa/j$ for some $\kappa>0$.
\begin{thm}
Assume $C \in \mathcal{L}(\Xs,\Ys)$ is such that the system is exactly observable at time $T$ and that $\epsilon$ is small enough so that the assumption of Lemma~\ref{lma:stable} is satisfied.
The back and forth observer with feedback $\frac{\kappa}je^{-\epsilon t} C^*$ converges to an estimate $\tilde x$, for which it holds that
\[
\norm{x^o-\tilde x} \le \frac{2\epsilon\norm{C}\sqrt{T}}{\alpha} \frac{2\sqrt{\lambda_1}}{\lambda_1-\epsilon^2/4} \norm{\chi}_{L^2(0,T)}
\]
where $\alpha$ is given in Lemma~\ref{lma:stable} and $\chi = y-Cz[x^o]$.
\end{thm}
\begin{proof}
Consider the error decomposition \eqref{eq:decomposition} with feedback term $\kappa_j k(t) C^*=\frac{\kappa}j e^{-\epsilon t} C^*$ at time $t=T$. The contribution of the second term in the decomposition is vanishing when $j \to \infty$ as seen in the proof of Theorem~\ref{thm:skew}, so let us concentrate on the first term, which can be opened up by substituting $k(s)I=P(s)-\big(P(s)-k(s)I \big)$:
\begin{align*}
&\frac{\kappa}j \int_0^T e^{A(T-s)}k(s)C^*\chi(s)ds \\ & \! =\frac{\kappa}j e^{AT} \! \int_0^T \! e^{A^*s}C^*\chi(s)ds +\frac{\kappa}j \int_0^T \! e^{A(T-s)}(k(s)I-P(s))C^*\chi(s)ds
\end{align*}
where the first term is zero by \eqref{eq:perpnoise}. An upper bound for the norm of the second term is given by $\frac{\epsilon\kappa}j \frac{2\sqrt{\lambda_1}}{\lambda_1-\epsilon^2/4}\norm{C}\!\sqrt{T} \norm{\chi}_{L^2(0,T)}=: \! {M\kappa}/j$, as in the proof of Theorem~\ref{thm:wave1}. The total contribution of these terms after $j$ iterations is bounded by 
\begin{equation} \label{eq:bdsum}
2M\sum_{i=1}^j \frac{\kappa}i \prod_{k=i+1}^j \left(1-\frac{\alpha\kappa}k \right).
\end{equation}
The product can be bounded as in the proof of Theorem~\ref{thm:skew},
\[
\prod_{k=i+1}^j \! \left(1-\frac{\alpha\kappa}k \right) \le \exp \left( \! -\alpha\kappa \! \sum_{k=i+1}^j \frac1k \right) \le \exp \big( \alpha\kappa( \ln i-\ln j)\big)=\left( \frac{i}{j} \right)^{\alpha\kappa}
\]
where the second inequality is obtained by comparing the sum with the integral of $1/x$. Similarly, the sum in \eqref{eq:bdsum} can be bounded by the integral of $x^{\alpha\kappa-1}$
\[
 \sum_{i=1}^j \frac{\kappa}{j^{\alpha\kappa}} i^{\alpha\kappa-1} \le \frac{\kappa}{j^{\alpha\kappa}} \int_0^{j+1} x^{\alpha\kappa-1}dx=\frac{1}{\alpha}\left( \frac{j+1}j \right)^{\alpha\kappa}.
\]
As $j \to \infty$, all other terms in $\varepsilon_j^-(T)$ tend to zero and so the result follows.
\end{proof}

\section*{Acknowledgment}

The author thanks Philippe Moireau for discussions concerning this work.



\end{document}